\documentclass{file}
\usepackage{pat}
\usepackage{paralist}
\usepackage{dsfont}  
\usepackage[utf8x]{inputenc}
\usepackage{skull}
\usepackage{paralist}
\usepackage{graphicx}
\usepackage{wrapfig}
\usepackage{subfig}
\usepackage{thm-restate}
\usepackage[shortlabels]{enumitem}
\usepackage[noend]{algpseudocode}
\usepackage{algorithm}
\makeatletter
\renewcommand{\ALG@name}{Procedure}
\makeatother
\usepackage{bbm}  
\usepackage{listings}
\usepackage{blindtext}
\usepackage{caption}
\usepackage{titlesec}
\usepackage[normalem]{ulem}
\usepackage{cancel}
\usepackage[noeepic]{qtree}
\usepackage{todonotes}

\usepackage{etoolbox}

\usepackage{float}
\usepackage{tabularx}
\usepackage{array}
\usepackage{multirow}
\usepackage{makecell}

\usepackage[longnamesfirst,numbers,sort&compress]{natbib}

\usepackage[mathlines]{lineno}
\newcommand*\patchAmsMathEnvironmentForLineno[1]{%
 \expandafter\let\csname old#1\expandafter\endcsname\csname #1\endcsname
 \expandafter\let\csname oldend#1\expandafter\endcsname\csname end#1\endcsname
 \renewenvironment{#1}%
    {\linenomath\csname old#1\endcsname}%
    {\csname oldend#1\endcsname\endlinenomath}}%
\newcommand*\patchBothAmsMathEnvironmentsForLineno[1]{%
 \patchAmsMathEnvironmentForLineno{#1}%
 \patchAmsMathEnvironmentForLineno{#1*}}%
\AtBeginDocument{%
\patchBothAmsMathEnvironmentsForLineno{equation}%
\patchBothAmsMathEnvironmentsForLineno{align}%
\patchBothAmsMathEnvironmentsForLineno{flalign}%
\patchBothAmsMathEnvironmentsForLineno{alignat}%
\patchBothAmsMathEnvironmentsForLineno{gather}%
\patchBothAmsMathEnvironmentsForLineno{multline}%
}

\newcommand{\vol}[1]{\mathsf{Volt}(}

\usetikzlibrary{trees,arrows}

\definecolor{brightmaroon}{rgb}{0.76, 0.13, 0.28}
\definecolor{linkblue}{rgb}{0, 0.337, 0.227}

\newcommand{\cay}{\mathop{\mathsf{Cay}}}

\newcommand{\cartprod}{\mathbin{\Box}}
\newcommand{\st}{\mathsf{St}}
\newcommand{\g}{\overline{G}}
\newcommand{\s}{\overline{S}}

\newcommand{\x}{\overline{x}}
\newcommand{\y}{\overline{y}}

\newcommand{\pp}{P_{\infty} \mathbin{\Box} P_{\infty}}
\newcommand{\np}{P_{n} \mathbin{\Box} P_{\infty}}

\newcommand{\crefdefpart}[2]{%
  \hyperref[#2]{\namecref{#1}~\labelcref*{#1}~\ref*{#2}}%
}

\newrobustcmd{\onesub}{\mathord{\includegraphics{figs/one-sub}}}
\newrobustcmd{\leftup}{\mathord{\includegraphics{figs/left-up}}}

\title{{Hamiltonicity of Transitive Graphs Whose Automorphism Group Has $\Z_{p}$ as Commutator Subgroups}}
\author{Florian Lehner\thanks{Department of Mathematics, University of Auckland, 38 princes street, 1010, Auckland, New Zealand} \,\,\,\,\,Farzad Maghsoudi\thanks{Department of Mathematics and Computer Science, University of Lethbridge, Lethbridge, AB, Canada.} \,\,\,\,\,Babak Miraftab\thanks{School of Computer Science, Carleton University, Ottawa, ON, Canada.} }

\date{}

\begin{document}

\maketitle

\begin{abstract}
In 1982, Durnberger proved that every connected Cayley graph of a finite group with a commutator subgroup of prime order contains a hamiltonian cycle. In this paper, we extend this result to the infinite case. 
Additionally, we generalize this result to a broader class of infinite graphs $X$, where the automorphism group of $X$ contains a transitive subgroup $G$ with a cyclic commutator subgroup of prime order.
\end{abstract}

\section{Introduction}

A \defin{hamiltonian path (cycle)} in a finite graph is a path (cycle) that traverses every vertex of the graph exactly once.
A graph $X$ is called \defin{vertex-transitive} if its automorphism group acts transitively on its vertices, in other words, for any two vertices $v_1$ and $v_2$ in $V(X)$, there exists an automorphism $f\colon V(X)\to V(X)$ such that $f(v_{1})=v_{2}$.
Lovász in 1969 conjectured that every finite connected vertex transitive graph has a hamiltonian path~\cite[p 497]{MR0263646}, but even decades later this conjecture remains largely unresolved. A particularly significant class of transitive graphs is the class of Cayley graphs. 
Lovász' conjecture has been verified for certain classes of Cayley graphs, such as Abelian groups, $p$-groups \cite{p-group}, and groups whose order has few prime factors \cite{pqrs}, see the surveys \cite{WitteGallian-survey, Lanel} for more results.

\noindent Hamiltonicity problems have also received significant attention in the context of infinite graphs, as evidenced by recent studies \cite{MR3799423,miraftab2022hamiltonicity,miraftab2023vertex,erdelehner,darijani2022arc}. It is worth noting that there are several appropriate generalizations of hamiltonian cycles for infinite graphs.
In this paper, our focus is on two-way hamiltonian paths.
\begin{defn}[cf.\ {\cite[pp.~286 and 297]{NashWilliams-Infinite}}] \label{2wayHPDefn}
Let $X$ be a countably infinite graph.
A \defin{two-way hamiltonian path} is a two-way infinite list $\ldots, x_{-2}, x_{-1}, x_{0}, x_{1}, x_{2}, \ldots$
containing each vertex of $X$ exactly once, such that $x_i$ is adjacent to $x_{i+1}$ for each $i$.
\end{defn}

\noindent Many results about hamiltonian paths in finite graphs carry over to the infinite setting. For instance, the following result is well-known.

\begin{lem}\label{abelian_ham}
    Let $G$ be a finitely generated abelian group.
    Then the following holds:
    \begin{enumerate}[\rm (i)]
        \item {\rm\cite[Corollary 3.2]{MR0708161}} If $G$ is finite, then every Cayley graph of $G$ has a hamiltonian cycle.
        \item{\rm\cite[Theorem 1]{MR0105438}} If $G$ is infinite, then every Cayley graph of $G$ has a two-way hamiltonian path. 
    \end{enumerate}
\end{lem}

Recall that the commutator subgroup $G'$ of a group $G$ is the subgroup generated by all commutators $ \langle [x,y]=xyx^{-1}y^{-1} \mid x, y \in G\rangle$, so $G$ is abelian if and only if the commutator subgroup of $G$ is trivial. Since Lovász' conjecture is known to be true for Cayley graphs of abelian groups, it seems natural to explore scenarios where the order of the commutator subgroup of $G$ is small.
In 1982, Durnberger proved the following:
\begin{thm}{\rm\cite[Theorem 1]{primeG'}}\label{finitezp}
    There is a hamiltonian cycle in every connected Cayley graph of a finite group whose commutator subgroup has prime order.
\end{thm}

\noindent Our first result is an extension of \cref{finitezp} to infinite locally finite Cayley  graphs.

\begin{restatable}{thm}{maintwo}
\label{main_2}
Let $G=\langle S\rangle$ with $G'\cong \Z_p$.
Then $\cay(G;S)$ has a two-way
hamiltonian path.
\end{restatable}

\noindent Then our next result is a generalization of the preceding theorem to infinite locally finite transitive  graphs.

\begin{restatable}{thm}{mainone}
\label{thm:mainone}
Let $X$ be an infinite locally finite $G$-transitive graph such that $G'\cong \Z_p$. Then $X$ has a two-way hamiltonian path.
\end{restatable}

\noindent {\bf An overview of the proof techniques:}
We first prove that every Cayley graph of a finitely generated group $G$, where $G'\cong \Z_p$, contains a two-way hamiltonian path. We then extend this result to infinite, locally finite, transitive graphs whose automorphism groups have a commutator subgroup that is cyclic of prime order.

The first step in the proof of Theorem \ref{main_2} is to find a two-way hamiltonian path $\gamma$ in a Cayley graph of $G/G'$ together with an infinite set of edges $F$ aligned with $\gamma$, see Figure \ref{twoended} for a sketch. More precisely, we define the following:
\begin{defn}
Let $X$ be a multigraph and let $\gamma\colon\ldots, v_{-1}, v_0, v_1, \ldots$ be a two-way-infinite  path in $X$.
Call an infinite sequence $\{e_{\ell}=(x_{\ell},y_{\ell})\mid \ell\in \Z\}$ of edges of $X$ \defin{aligned} with $\gamma$ if none of the edges $e_\ell$ is contained in $\gamma$, and for every $\ell \in \mathbb Z$ there is some $j \in \mathbb Z$ and $k \geq 1$ such that $x_\ell = v_j$, $y_\ell = v_{j+k}$, and $x_{\ell+1} = v_{j+k+1}$. In this case, each edge $e_\ell$ is called \defin{jumping}, see \cref{twoended}.
\end{defn}
We note that the case $k = 1$ in the above definition can only occur in multigraphs, since in this case the edge $e_\ell$ is parallel to an edge in the infinite path $\gamma$.

\begin{figure}
    \centering
    \includegraphics{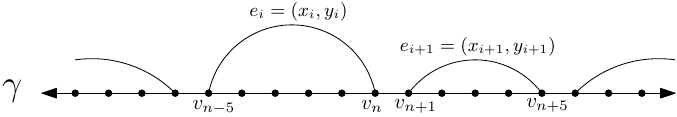}
    \caption{An example of a two-way infinite path aligned with $[e_{\ell}\mid \ell\in \Z]$.}
    \label{twoended}
\end{figure}

\begin{nota}
    If $\gamma$ is a two-way infinite path aligned with a sequence $[e_{\ell}=(x_{\ell},y_{\ell})\mid \ell\in \Z]$,
    then adding an edge $e_{\ell}$ to $\gamma$ gives rise to a unique cycle. We denote this cycle by $C_{\ell}$.
\end{nota}

The second step in our strategy is to use an aligned sequence of edges to ``lift'' a hamiltonian path in a Cayley graph of $G/G'$ to a hamiltonian path in a Cayley graph of $G$. To simplify notation, we introduce the following notation, where usually $N$ will be equal to~$G'$.

\begin{nota}
    Let $G$ be a group with a normal subgroup $N$.
    Then there is a projection homomorphism $\pi\colon G\to G/N$.
    For every subgroup(subset) of $H$ of $G$, we denote the image of $H$ under $\pi$ by $\overline H$; for every $g \in G$ we denote the coset $gN$ by $\overline{g}$.
\end{nota}

Let's return to the strategy. 
Since we assume that $G'$ is finite, each cycle $C_\ell$ in the Cayley graph $G/G'$ lifts to a finite subgraph of the Cayley graph of $G$, referred to as a block ($\pi^{-1}(C_\ell)$), see \Cref{def_block}.

The trickiest part of the proof is finding a hamiltonian path with an aligned sequence that ensures that each block has a hamiltonian cycle. To this end, we use the concept of voltage.

\begin{defn}[cf.~{\cite{Ebrahim}}]
    Let $ G = \langle S \rangle $ be a group, where $S$ is finite, and let $N$ be a normal subgroup of $G$. 
    Let  $C\coloneqq \overline{g}(\overline{s_1},\ldots,\overline{s_n})$ be a cycle in $\cay(\overline{G};\s)$.
    The \defin{voltage} of $C$ denoted by $\vol(C)$ is defined as the product $s_1\cdots s_n$.
\end{defn}

\noindent It can be shown (see \Cref{hc_in_blocks}) that if the voltage of $C_\ell$ is non-trivial, then the corresponding block has a hamiltonian cycle.

Finally, if every block has a hamiltonian cycle, we can combine suitable hamiltonian paths within each block to form a two-way hamiltonian path in the Cayley graph of $G$, see \Cref{enter-label}. 

Next, let $X$ be a locally finite graph, where a group $ G $ acts transitively on $ X $ and $ G' = \mathbb{Z}_p $. 
We leverage the existence of a two-way hamiltonian path in a Cayley graph to construct a two-way hamiltonian path in $X$.
More specifically, we consider the quotient graph $ X/G' $, which is a Cayley graph of an abelian group. A two-way hamiltonian path in the quotient is lifted to $ X $, forming paths in $X$ corresponding to the cosets of $ G' $. 
As in the previous case, blocks are defined in $ X $, and we show that each block contains a hamiltonian path. These paths are then connected across blocks via a matching, resulting in the desired two-way hamiltonian path.


\section{Preliminaries}
Throughout this paper, we assume that  $G$  is an infinite group with a finite generating set  $S$  that is symmetric and minimal.
When considering Cayley graphs we always assume them to be connected.


\subsection{Brief review on quotient graphs}

\begin{defn} [cf.~{\cite[\S2 pp 4] {DicksandDunwoody}}]
    A \defin{$G$-graph} is a graph $X$ admitting an injective homomorphism $\varphi\colon G \to \mathsf{Aut}(X)$.
    If the action $G$ on $X$ is transitive, then we say $X$ is a \defin{transitive $G$-graph}.
\end{defn}
In this paper, we assume all $ G $-actions are faithful, meaning there are no group elements $ g \in G $ (other than the identity element) such that $ g x = x $ for all $ x $ in $ X $.




\begin{defn}[{\cite[Definition 2.8.]{transitive-ham}}]
Let $X$ be a $G$-graph and let $H$ be a subgroup of $G$. 
The \defin{quotient graph} $X/H$ is that graph whose vertices are the
$H$-orbits, and two such vertices $Hx$ and $Hy$ are adjacent in $X/H$ if and only if there is an edge in $X$ joining a vertex of $Hx$ to a vertex of $Hy$. 
\end{defn}
\begin{obs}
If $X$ is a transitive $G$-graph and  $H$ is normal in $G$, then $X/H$ is vertex-transitive: the action
of $G$ on $V (X)$ factors through to a transitive action of $G/H$ on $V (X/H)$, by automorphisms of $X/H$.  
\end{obs}

\begin{nota}
Let $X$ be a $G$-graph and let $N$ be a normal subgroup of $G$.   
Then there is a canonical homomorphism $\pi\colon X\to X/N$.
In particular, if $X=\cay(G;S)$, then there is a graph homomorphism $\pi\colon \cay(G;S)\to \cay(G/N,SN)$.
\end{nota}

\begin{defn}
Let $X$ be a $G$-graph and let $H$ be a subgroup of $G$. 
Let $e\in E(X/H)$. Then an edge $(u',v')$ of $X$ is called a \defin{lifting} of $e$ if $\pi((u',v'))=e$.
A lifting of a subgraph is defined analogously.
\end{defn}

Our next goal is to show that given a $G$-graph $X$ where $G' = \mathbb Z_n$, and a hamiltonian path $\gamma$ in $X/G'$ we can cover $X$ with liftings of $\gamma$ each of which is a two-way infinite path. The following lemma is well-known, but we provide a proof for the convenience of the reader.

\begin{lem}\label{lift_edge}
Let $X$ be a $G$-graph and let $H$ be a subgroup of $G$.
For every edge $(u,v)\in E(X/H)$ and every $x\in \pi^{-1}(u)$, there is an edge $(x,y)\in E(X)$ such that $\pi((x,y))=(u,v)$.
\end{lem}

\begin{proof}
By definition, there exist adjacent vertices $x'\in \pi^{-1}(u)$ and $y'\in \pi^{-1}(v)$. Since $x'\in \pi^{-1}(u)$ and $x\in \pi^{-1}(u)$, there exists $h\in H$ such that $hx'=x$, implying that $x$ is adjacent to $hy'$.
\end{proof}

\begin{lem}[{\cite[Corollary 2.5]{transitive-ham}}]\label{Gcapst=1} 
Let $X$ be a transitive $G$-graph with a vertex $v$.
Then the stabilizer of $v$ i.e. $\st_G(v)$, does not contain a non-trivial, normal subgroup of $G$.
\end{lem}

\begin{lem}\label{|N|-liftings}
Let $X$ be a $G$-graph, where $G'=\Z_{n}$.
Then for every  two-way hamiltonian path $\overline{\gamma}$ of $X/G'$, there is a set $L(\overline{\gamma})\coloneqq \{\gamma_1,\ldots,\gamma_{n}\}$ of liftings of $\overline{\gamma}$  such that each $\gamma_i$ is a two-sided infinite path, and
\[V(X)=\bigsqcup_{i=1}^{n} V(\gamma_i).\]
\end{lem}

\begin{proof}
Let $ G' = \{g_1,g_2,\ldots,g_{n}\}  $, where $ g_1 = 1 $, and $ \overline{\gamma} $ be a two-way hamiltonian path in $X/G' $ with the following lifting:
$ \ldots, v_{-2}, v_{-1}, v_0, v_{1}, v_{2}, \ldots$
Also, let $ \gamma_i $ be a lifting of $ \overline{\gamma} $ in $ G $ as follows:
$\ldots, g_iv_{-2}, g_iv_{-1}, g_iv_0, g_iv_{1}, g_iv_{2}, \ldots$
where $ i \in \{1,2,\ldots,n\} $. 
Next, we show that they are vertex-disjoint. 
Assume to the contrary that $g_iv_{n'}=g_jv_{m'}$.
Thus, we have $v_{n'}=g_i^{-1}g_jv_{m'}$ which implies that $G'[v_{n'}]=G'[v_{m'}]$.
If $v_{n'}\neq v_{m'}$, then it yields a contradiction, as $\overline \gamma$ is a hamiltonian two-way path in $X/G'$.
So we infer that $v_{n'}=v_{m'}$ which implies that $g_i^{-1}g_j\in \st(v_{n'})$.
If $g_i\neq g_j$, then $(G'\cap \st_G(v_{n'})) \,\mathsf{char }\, G' \unlhd \,G$ which implies that $G'\cap \st_G(v_{n'})$ is a normal subgroup of $G$.
It now follows from \Cref{Gcapst=1} that $\st_G(v_{n'})$ has a non-trivial normal subgroup which yields a contradiction and so~$g_i=g_j$.
\end{proof}

\noindent Let $G=\langle S\rangle $ be a group with a finite commutator subgroup.
Let $\cay(\g;\s)$ have a two-way hamiltonian path $\gamma$ aligned with sequence $[e_{\ell}\mid \ell\in \Z]$. 
Every jumping edge $e_{\ell}$ defines a block of vertices in $\cay(G;S)$ in the following way:

\begin{defn}\label{blocks}
Let $X$ be a transitive $G$-graph, where $|G'|<\infty$.
Let $\gamma\colon \ldots,v_1,v_0,v_1,\ldots$ be a two-way hamiltonian path in $X/G'$ aligned with a sequence $[e_{\ell}\mid {\ell}\in \Z]$.
Then the \defin{block $B((v_{i_n},v_{j_n}))$} is the  induced  subgraph of $X$ on the  vertices $\bigcup_{i_n\leq i\leq j_n} \pi^{-1} (v_{i})$, where $e_{n}=(v_{i_{n}},v_{j_{n}})$.
\begin{figure}
    \centering
    \includegraphics[scale=0.6]{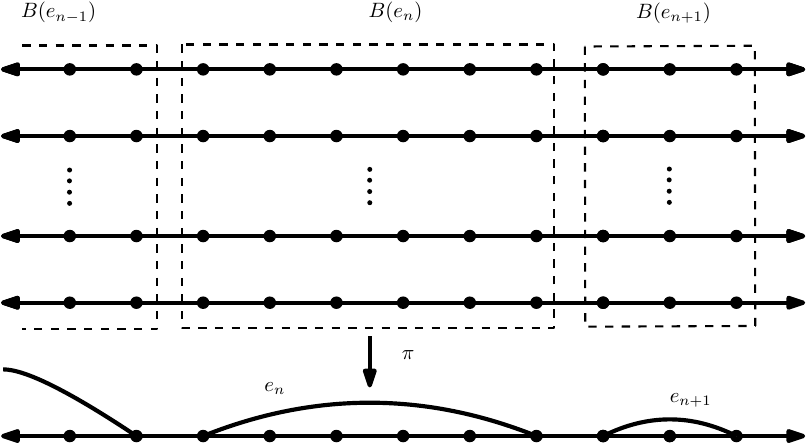}
    \caption{The Block is defined by a two-way hamiltonian path aligned with a sequence.}
    \label{def_block}
\end{figure}
\noindent The \defin{entrance} of $B(e_n)$ are the set of vertices $\pi^{-1}(v_{i_n})$ and the \defin{exit} of $B(e_{n})$ are the set of vertices $\pi^{-1}(v_{j_n})$. 
\end{defn}

\begin{defn}
Let $X$ be a $G$-graph with an edge $(u,v)\in E(X/G')$ and let $G'=\{g_1,\ldots,g_n\}$.
A lifting $(g_iu,g_jv)$ of $(u,v)$ is called \defin{bouncing edge} if $i\neq j$.
\end{defn}

\begin{rem}\label{bounce-vol}
It is not hard to see that an edge $e_\ell$ is bouncing if and only if $\vol(C_\ell)\neq 1$.
We note that if a lifting of an edge $e_\ell$ is bouncing, then there is a bouncing edge adjacent to each lifting of an endpoint of $e_\ell$.
\end{rem}

\begin{lem}\label{hc_in_blocks}
Let $X$ be an infinite transitive $G$-graph with $G'\cong \Z_{p}$ for some prime $p$.
Let $\gamma$ be a two-way hamiltonian path in $X/G'$ aligned with a sequence $[e_{\ell}\mid {\ell}\in \Z]$ such that each $e_{\ell}$ is a bouncing edge.
Every block of $X$ has a hamiltonian cycle containing some bouncing edge.
\end{lem}

\begin{proof}
Let $ G' = \langle g\rangle  $,  and let $ \ldots, \overline v_{-2}, \overline v_{-1}, \overline v_0, \overline v_{1}, \overline v_{2}, \ldots$ be a two-way hamiltonian path in $X/G'$ with a lifting
$\ldots, v_{-2}, v_{-1}, v_0, v_{1}, v_{2}, \ldots$.
Without loss of generality, we may assume that the block $B(e_0)$ is $\pi^{-1}(\{\overline v_0, \dots, \overline v_{k}\})$ and that there is a bouncing edge $(v_0,g v_k)$.

For $ i \in \{0,1,2,\ldots,p-1\}$ there is a lifting $ \gamma_i $  of the two-way hamiltonian path in $X/G'$ defined as follows:
$\ldots, g^iv_{-2}, g^iv_{-1}, g^iv_0, g^iv_{1}, g^iv_{2}, \ldots$. By \Cref{|N|-liftings} the lifts $\gamma_i$ are pairwise disjoint.

We observe that the vertices of the block $B(e_0)$ are $\bigcup_{i=1}^{n} \{g^iv_0,g^iv_1 \ldots, g^iv_{k}\}$, and that there exists a subpath $p_i$ of $\gamma_{i}$ connecting $g^iv_0$ to $g^iv_{k}$).
We further note that $(g^iv_0 ,g^{i+1}v_{k})$ is a bouncing edge and that each bouncing edge $(g^iv_0, g^{i+1}v_{k})$ connects the first vertex of $p_i$ to the last vertex of $p_{i+1}$.

It follows from our construction that the edges of all paths $p_i$ together with the bouncing edges $(g^iv_0, g^{i+1}v_{k})$ form a hamiltonian cycle, see \cref{enter-label}.
\end{proof}

\begin{figure}[H]
    \centering
    \includegraphics[scale=0.8]{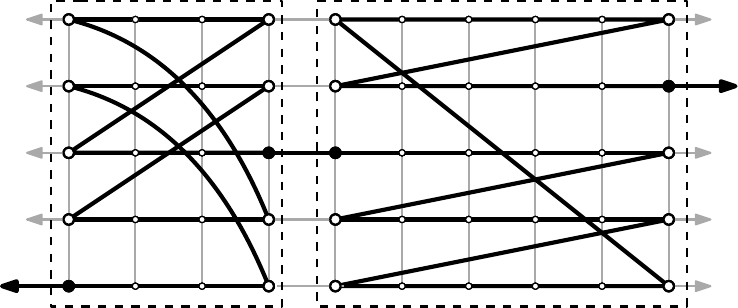}
    \caption{Connecting hamiltonian cycles from each block to a two hamiltonian path. The entrance and exit of each block are indicated by black vertices.}
    \label{enter-label}
\end{figure}

\begin{cor}\label{blocks_have_HC}
Let $B(e_{\ell})$ be an arbitrary block, and let $u$ be at the entrance of $B(e_{\ell})$.
Then there exists a hamiltonian path in $B(e_{\ell})$ starting with $u$ and ending at a vertex in the exit of $B(e_{\ell})$.
\end{cor}


\section{Aligned sequences with two-way paths}\label{long}
Let $G$ be an infinite finitely generated, non-abelian group whose commutator subgroup $G'$ has infinite index. Let $S$ be some finite generating set of $G$ and let $\s$ be the corresponding generating set of $\g = G/G'$. We will show that under these assumptions we can find a two-way hamiltonian path in $\cay(\g;\s)$ aligned with a sequence of edges such that the corresponding cycles have non-trivial voltages.

Recall the definition of the Cartesian product of two graphs.

\begin{defn}[cf.~{\cite[\S7.4]{Godsil}}] \label{CartProdDefn}
The \defin{Cartesian product} $X \cartprod Y$ of two graphs $X$ and $Y$ is the graph with vertex set $V(X) \times V(Y)$ where $(x_1, y_1)$ is joined to $(x_2, y_2)$ by an edge if and only if either $x_1 = x_2$ and $(y_1, y_2) \in E(Y)$ or $y_1 = y_2$ and $(x_1, x_2) \in E(X)$.
\end{defn}

Our general strategy is as follows. Since $G'$ has infinite index, $\g$ is an infinite abelian group. Therefore $\cay(\g;\s)$ has a spanning subgraph isomorphic to some grid $\pp$ or $\np$, where $P_n$ denotes the path of length $n$ and $P_{\infty}$ denotes a two-way-infinite path. In these spanning subgraphs, we can find hamiltonian paths with aligned sequences of edges, see Figure \ref{aligned+volt}. Finally, we slightly modify the hamiltonian paths to ensure that all cycles corresponding to jumping edges have non-trivial voltages.

In order for this last step to work, we need some control over the generators corresponding to the various edges in our grids. In \Cref{goodgrids} we define precisely what this means and show that we can always find spanning grids that meet our requirements. 

\begin{figure}
\centering
\subfloat[\centering ]{{\includegraphics[scale=0.50]{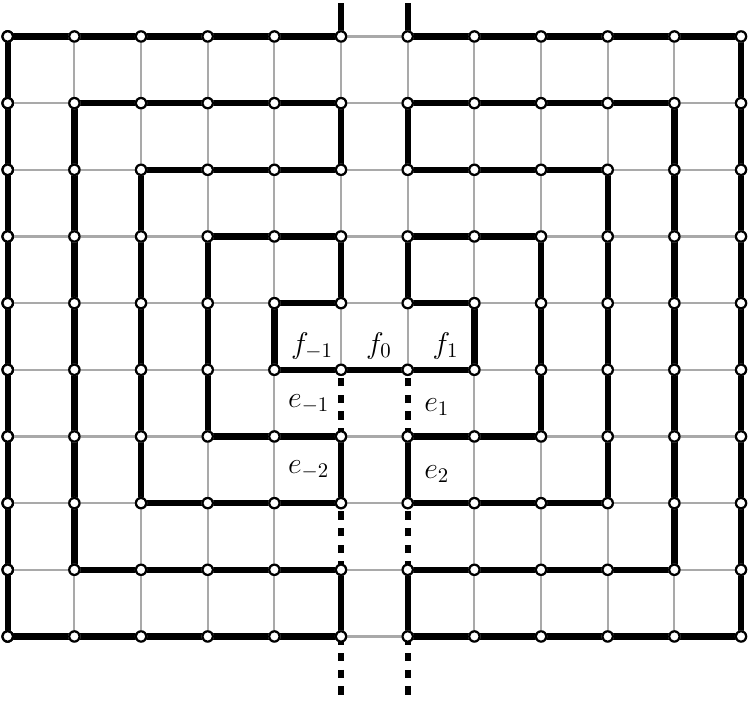}}}
\qquad
\subfloat[\centering ]{{\includegraphics[scale=0.50]{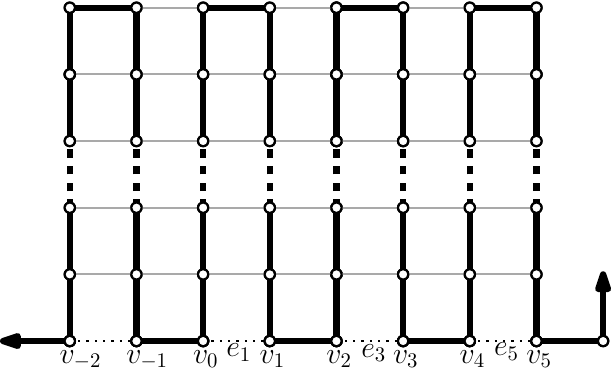} }}%
\caption{Two-way hamiltonian path and jumping edge $ e_i $ in $\pp$ and $\np$}
\label{aligned+volt}
\end{figure}

\subsection{Constructing spanning grids}
\label{goodgrids}

The goal of this section is to find useful spanning substructures in Cayley graphs of abelian groups. Later we will apply the results to the group $G/G'$, but throughout this subsection, we let $G$ be an abelian group with generating set $S$.

We start by describing the substructures we aim to find. 
Let $P_n$ denote the path of length $n$ and let $P_{\infty}$ denote the two-way infinite path.
A \defin{spanning grid} of $\cay (G;S)$ is a subgraph of $\cay(G;S)$ isomorphic to $P_{m} \cartprod P_n $ containing all vertices of $\cay (G;S)$ where $m,n \in \mathbb N \cup \infty$. We refer to the edges of the grid corresponding to $P_{m}$ as \defin{horizontal} edges and the edges corresponding to edges of $P_n$ as \defin{vertical} edges.

Recall that edges in a Cayley graph can be labelled by the corresponding generators. Call a grid in $\cay (G;S)$ \defin{consistently labelled} if all horizontal edges in the same column have the same label, and all vertical edges in the same row have the same label. We say that a consistently labelled grid has an \defin{$X$-component} for $X \subseteq S$ if either all horizontal edges and none of the vertical edges are labelled with elements of $X$, or all vertical and none of the horizontal edges are labelled $X$. We call $X$-components finite or infinite depending on whether the corresponding factor in the Cartesian product is finite or infinite. To simpify notation, we write $x$-component instead of $\{x\}$-component and $x$-$y$-component instead of $\{x,y\}$-component. Finally, let us call a path of length $3$ a \defin{useful $x$-triple}, if its first and last edges correspond to generators $x^{\pm 1}$ and $x^{\pm 1}$, and a \defin{strongly useful $x$-triple} if one of them corresponds to $x$ and the other to $x^{-1}$.

In what follows, by a hamiltonian path we mean either a two-way infinite hamiltonian path or (in case the corresponding graph is finite) a standard finite hamiltonian path.
\begin{lem}
\label{spanning-single}
Let $G$ be an abelian group with a finite generating set $S$.
For every $x \in S$ the Cayley graph $\cay (G,S)$ contains one of the following: 
\begin{itemize}
    \item a hamiltonian path all of whose edges correspond to the generator $x$, or
    \item a consistently labelled spanning grid with an $x$-component. 
\end{itemize} 
\end{lem}

\begin{proof}
    If $G= \langle x \rangle$, then $G$ is cyclic and the edges corresponding to $x$ form a hamiltonian path. If $G \neq \langle x \rangle$, then we can find a hamiltonian path in the Cayley graph of $G/ \langle x \rangle$ by \Cref{abelian_ham}. Let $\gamma$ be a lift of this hamiltonian path in $\cay(G;S)$.

    If the index of $\langle x \rangle$ in $G$ is infinite, then $x^i \gamma$ is also a lift of the hamiltonian path, and all these lifts are disjoint (compare \Cref{|N|-liftings}). By adding all edges labelled $x$ to the union of all of these lifts we obtain a consistently labelled spanning grid with an infinite $x$-component.

    If $\langle x \rangle$ has finite index $k$, then $x^i \gamma$ is a lift of the hamiltonian path for $0 \leq i < k$. By adding all edges labelled $x$ which connect $x^i \gamma$ to $x^{i+1} \gamma$ for $0 \leq i < k-1$, we obtain a consistently labelled spanning grid with a finite $x$-component.
\end{proof}

\begin{lem}
    \label{usefultriples}
    Let $G$ be a (finite or infinite) abelian group with a generating set $S$, let $x \in S$, and assume that $S$ contains another element $y \neq x^{\pm 1}$. Then $\cay(G;S)$ contains a hamiltonian path containing a useful $x$-triple. If $\langle x \rangle$ is finite, we can find a hamiltonian path containing a strongly useful $x$-triple.
\end{lem}

\begin{proof}
    First, assume that $x$ does not generate $G$. By \Cref{spanning-single}, we can find a spanning grid with an $x$-component. If $\langle x \rangle$ is finite, then the $x$-component of this grid is finite, and the hamiltonian path sketched in \Cref{aligned+volt} (b) contains a strongly useful $x$-triple. So we may assume that $\langle x \rangle$ is infinite and thus the $x$-component of the grid is infinite. If the $S \sm \{x^{\pm 1}\}$-component is infinite or only consists of a single edge, then the hamiltonian paths shown in \Cref{aligned+volt} (a) have useful $x$-triples. If the $S \sm \{x^{\pm 1}\}$-component is finite but consists of more than one edge, then we can modify the hamiltonian path sketched in \Cref{aligned+volt} (b) locally to obtain a useful $x$-triple, see \cref{triple_x}.
    \begin{figure}
        \centering
        \includegraphics[scale=0.5]{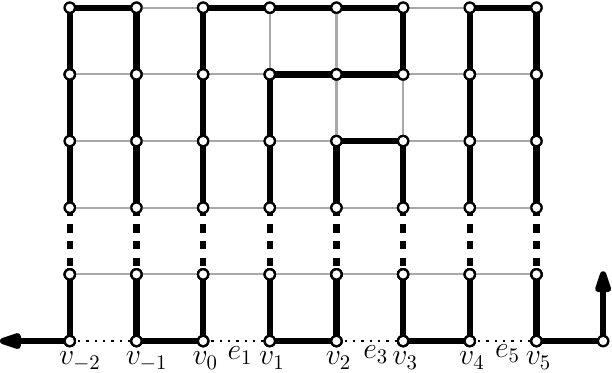}
        \caption{Modifying the the two-way hamiltonian path of \cref{aligned+volt}.}
        \label{triple_x}
    \end{figure}

    Now assume that $G = \langle x \rangle$. Let $i$ be the minimal value such that $x^i = y$. Take the hamiltonian path all of whose edges correspond to $x$. Since $y \neq x^{-1}$ we know that this path contains at least $i+1$ edges. Now replace some subpath of the form $(x^{i+1})$ by $(x,y,x^{-i+1},y)$. It is easy to see that this is again a hamiltonian path, and it clearly contains a strongly useful $x$-triple.
\end{proof}

\begin{lem}
\label{spanning-double}
Let $G$ be an infinite abelian group with a generating set $S$.
For every pair $x, y \in S$ of distinct generators,
if $G\neq\langle x,y\rangle$, then $\cay (G,S)$ contains a consistently labelled spanning grid with one of the following (up to swapping the roles of $x$ and $y$): 
\begin{itemize}
    \item an infinite $x$-component and an infinite $S \sm \{x^{\pm 1 }\}$-component with a useful $y$-triple, or
    \item an infinite $x$-component and a finite $S \sm \{x^{\pm 1}\}$-component with a strongly useful $y$-triple, or
    \item a finite $x$-$y$-component and a strongly useful $y$-triple in this $x$-$y$-component.
\end{itemize} 
\end{lem}

\begin{proof}
    We proceed similarly to the proof of \Cref{spanning-single}. Let $H = \langle x,y \rangle$.  Note that by assumption $G \neq H$. 
    
    Assume first that $H$ is finite. By \Cref{usefultriples}, we can find a hamiltonian path $v_1, \dots, v_n$ in $\cay(H,\{x,y\})$ containing a strongly useful $y$-triple. As in the proof of \Cref{spanning-single}, find a hamiltonian path in the Cayley graph of $G/H$ and let $\gamma$ be a lift of this hamiltonian path in $\cay(G;S)$. Then $v_i \gamma$ is also a lift of this hamiltonian path and the lifts are disjoint. Now adding the edges with label $v_{i+1}v_i^{-1}$ between $v_i \gamma$ and $v_{i+1} \gamma$ for $1 \leq i < n$ gives a consistently labelled spanning grid with a finite $x$-$y$-component and a strongly useful $y$-triple in this $x$-$y$-component.

    Now assume that $H$ is infinite, and thus without loss of generality,  $\langle x \rangle$ is infinite. Further, without loss of generality assume that $y \notin \langle x \rangle$; if this is not the case, then $y$ also has infinite order and it is not hard to see that $x \notin \langle y \rangle$, so we only need to swap the roles of $x$ and $y$.

    By \Cref{usefultriples}, the Cayley graph of $G / \langle x \rangle$ contains a hamiltonian path with a useful $y$-triple, and this useful $y$-triple is strongly useful if $G / \langle x \rangle$ is finite. Like in the previous cases, taking lifts of this path and connecting them by edges corresponding to the generator $x$ like yields the desired grid.
\end{proof}




\subsection{Constructing Two-way hamiltonian paths in \texorpdfstring{$G/G'$}{}}

Throughout this section, let $G$ be a finitely generated non-abelian group and let $S$ be a generating set of $G$. Let $G'$ denote the commutator subgroup of $G$ and let $\g = G/G'$ with generating set $\s = \{\overline s \mid s \in S\}$. Let $x,y \in S$ such that $[x,y] \neq 1$.

Our aim in this section is to find a hamiltonian path in $\cay(\g;\s)$ with an aligned sequence $(e_\ell)_{\ell \in \mathbb Z}$ of edges such that the cycle corresponding to each $e_\ell$ has non-trivial voltage. We start with the easy case where $\x = \y$.

\begin{lem}
    \label{align-coincide}
    With the above notation, if $\x = \y$, then $\cay(\g;\s)$ contains a two-way hamiltonian path with an aligned sequence $(e_\ell)_{\ell \in \mathbb Z}$ such that the cycle corresponding to each $e_\ell$ has non-trivial voltage.
\end{lem}

\begin{proof}
    We slightly abuse notation by treating $\x$ and $\y$ as different generators, and instead of each edge labelled $\x$ drawing a pair of parallel edges labelled $\x$ and $\y$, respectively. By \Cref{spanning-single}, the Cayley graph $\cay(\g;\s)$ either has a hamiltonian path all of whose edges are labelled $\x$, or a consistently labelled spanning grid with an $\x$ component.

    In the first case, we can simply let  $(e_\ell)_{\ell \in \mathbb Z}$ consist of every second edge labelled $\y$ along the hamiltonian path. The corresponding cycles have length 2 and labels $x$ and $y$ (or their inverses). Since $y \neq x^{\pm 1}$, the voltage of each such cycle is non-trivial.

    In the second case, take a hamiltonian path in $\cay(\g;\s)$ with an aligned sequence of edges as sketched in \Cref{aligned+volt}. If the voltage of a cycle corresponding to an edge $e_\ell$ is trivial, replace one edge labelled $\x$ in this cycle by its parallel mate with label $\y$ (or vice versa). Since $y \neq x^{\pm 1}$, this changes the voltage of this cycle.
\end{proof}

From now on we will assume that $\x \neq \y$. Hence either $\g = \langle \x, \y \rangle$, or $\cay (\g;\s)$ contains one of the three spanning grids described in \Cref{spanning-double}. 

We first consider the case where the spanning grid is $P_{\infty} \cartprod  P_{\infty}$. We start with the hamiltonian path sketched in \Cref{aligned+volt}(a). Our aim is to recursively modify this hamiltonian path to ensure that the voltages of the cycles corresponding to jumping edges are non-trivial. Since the figure exhibits symmetry, it is enough to define the modifications for the cycles on the right side. 

All our modifications are confined to an `eighth' of the grid. The modification is sketched in \Cref{config}; note that if we take a hamiltonian path which contains every $p_i$ as a subpath, and replace every $p_i$ in \Cref{config} (a) by the corresponding $p_i'$ in \Cref{config} (b), then the result is again a hamiltonian path.

\begin{figure}
    \centering
    \subfloat[\centering ]{{\includegraphics[scale=0.56]{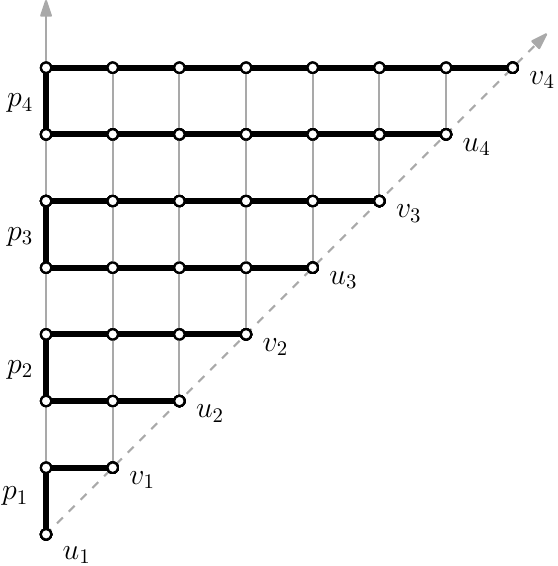} }}%
    \qquad
    \subfloat[\centering ]{{\includegraphics[scale=0.5]{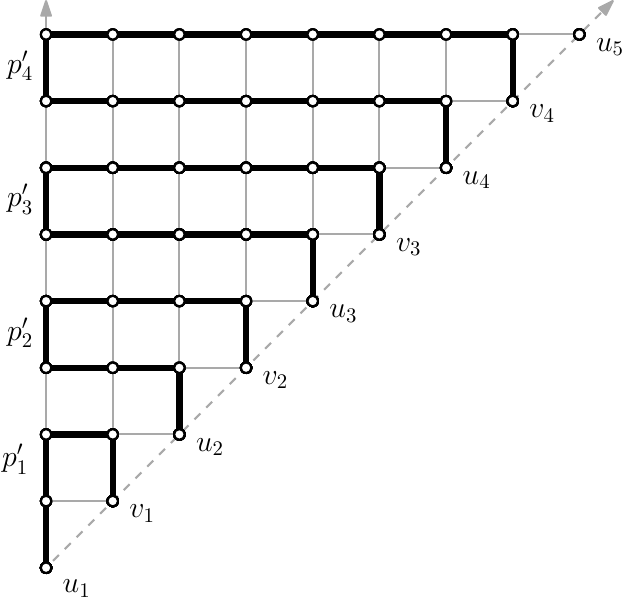} }}%
    \caption{Modifying the paths $p_i$ to obtain $p_i'$s}%
    \label{config}
\end{figure}

\begin{lem}
    \label{align-fullgrid}
    Let $\Gamma$ be a consistently labelled spanning grid of $\cay (\g;\s)$ with an infinite $\x$-component and an infinite $\s \sm \{\x^{\pm 1}\}$-component containing a useful $\y$-triple. Then $\cay (\g;\s)$ contains a two-way hamiltonian path with an aligned sequence $(e_\ell)_{\ell \in \mathbb Z}$ such that the cycle corresponding to each $e_\ell$ has non-trivial voltage.
\end{lem}

\begin{proof}
Let $\gamma$ be the two-way hamiltonian path in \Cref{aligned+volt}(a), and assume that all vertical edges are labelled $\x$, and that the three horizontal edges $f_{-1}, f_0, f_1$ connecting the central columns form a useful $\y$-triple. Our goal is to modify this hamiltonian path so that the voltage of the cycle corresponding to each jumping edge is non-trivial. To this end, we recursively apply the modification described above. 
We notice again that it suffices to only consider the right half-grid; the left half follows by symmetry.

Let us denote the $i$-th jumping edge by $e_i$, let $P_i$ be the subpath of $\gamma$ connecting the two endpoints of $e_i$, and let $C_i$ be the cycle consisting of $P_i$ and $e_i$. Let $P_{i,j}$ and $C_{i,j}$ be the corresponding path and cycle after $j$ steps of our modification procedure.
We will ensure that $\vol(C_{j,j}) \neq 1$, and that $C_{i,j} = C_{i,i}$ for every $j > i$. Moreover, our procedure ensures that for every $j$, each vertex of $C_i$ is contained in $C_{i',j}$ for some $i' \leq i$.

Before showing how to ensure these properties, we show that they allow us to define the desired hamiltonian path. Let $\gamma'$ be obtained from $\gamma$ by replacing every $P_i$ by $P_{i,i}$. Then $\gamma'$ is connected because $\gamma$ was connected, and every $P_{i,i}$ is connected. Moreover $\gamma'$ contains every vertex of each $P_i$, because each such vertex is contained in $P_{i',i} = P_{i',i'}$ for some $i' \leq i$. Next note that every vertex has degree $2$ in $\gamma$ because in order to have larger degree, a vertex would have to be contained in $P_{i,i}$ and $P_{i',i'}$ for $i > i'$, and thus have degree larger than $2$ in the $i$-th iteration. Hence $\gamma'$ is a hamiltonian path. The $e_i$ are still jumping edges for $\gamma'$, and $\vol(C_{i,i}) \neq 1$ by construction.

It remains to show that we can guarantee the claimed properties after each iteration. The properties clearly hold after $0$ iterations, as there is no $C_{i,i}$ whose voltage needs to be non-trivial. Inductively assume that they hold after $j-1$ iterations for some $j > 0$.

If $\vol(C_{j,j-1}) \neq 1$, we can set $C_{i,j} = C_{i,j-1}$, and all properties are satisfied after the $j$-th iteration. In preparation of the $(j+1)$-th step, we shift the eighth grid in which future modifications are to take place up by two steps, see \cref{oneeighth}.

\begin{figure}
    \centering
    \subfloat[\centering ]{{\includegraphics[scale=0.62]{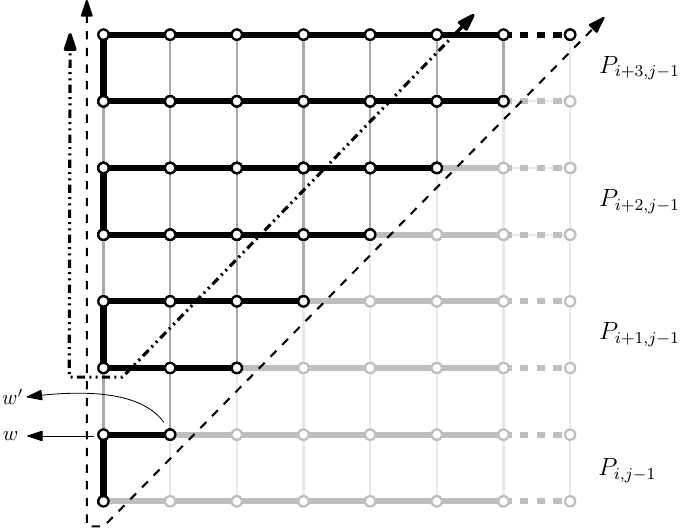} }}%
    \qquad
    \subfloat[\centering ]{{\includegraphics[scale=0.62]{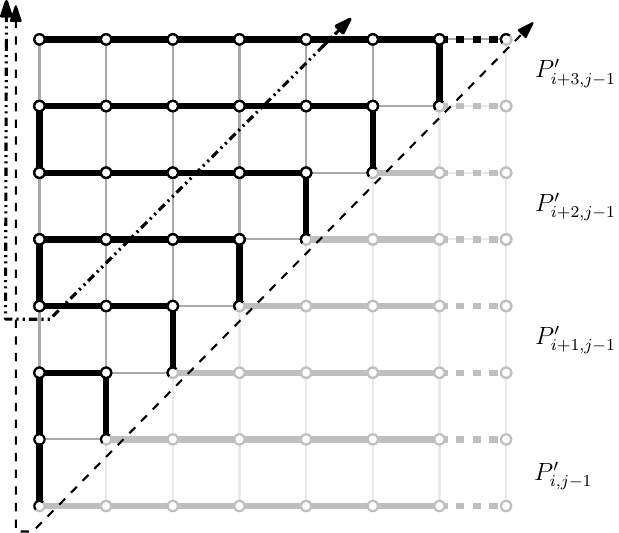} }}%
    \qquad
    \caption{Modifying the paths $P_{i,j-1}$ to obtain paths $P_{i,j-1}'$.}
    \label{oneeighth}
\end{figure}

So we may assume that the voltage of $C_{j,j-1}$ is trivial.
We set $P_{i,j} =  P_{i,j-1}$ for $0<i<j$, and $P_{i,j} =  P_{i,j-1}'$ as illustrated in \Cref{oneeighth} (b) for $i \geq j$. 
Finally, in preparation of the $(j+1)$-th step, we shift the eighth grid in which future modifications are to take place up by three steps, see \cref{oneeighth}.

For $i<j$ it holds that $C_{i,j} =  C_{i,j-1} = C_{i,i}$ by definition and the induction hypothesis. Further note that for arbitrary $i$, each vertex of $C_{i,j-1}$ is either contained in $C_{i,j}$ or in $C_{i-1,j}$, and thus by the induction hypothesis, each vertex of $C_i$ is contained in $C_{i',j}$ for some $i' \leq i$ as claimed.

It only remains to show that the voltage of $C_{j,j}$ is not trivial. Let $W$ denote the product of the labels along the subpath of $C_{j,j}$ from the initial vertex of $e_j$ to $w'$, as shown in \Cref{oneeighth} (a). Similarly, let $W'$ be the product of the labels along $C_{j,j}$ from $w$, also depicted in \Cref{oneeighth} (a), to the initial vertex of $e_i$. Let $r \in \{y,y^{-1}\}$ be the generator corresponding to the edge from $w$ to $w'$.
Clearly, $\vol(C_{j,j-1})= Wr^{-1}W' = 1$ which implies that $r=W'W$. 
\[
\vol(C_{j,j})=Wxr^{-1}x^{-1}W' = W'^{-1} W' Wxr^{-1}x^{-1}W' = W'^{-1} [r,x] W'
\]
Which is non-trivial because $[r,x] \neq 1$
%
%
%
\end{proof}

\begin{lem}
    \label{align-finitegrid}
    Let $\Gamma$ be a consistently labelled spanning grid of $\cay (\g;\s)$ with 
    \begin{itemize}
    \item an infinite $x$-component and a finite $S \setminus \{x\}$-component with a strongly useful $y$-triple, or
    \item a finite $x$-$y$-component and a strongly useful $y$-triple in this $x$-$y$-component.
    \end{itemize}
    Then $\cay (\g;\s)$ contains a two-way hamiltonian path with an aligned sequence $(e_\ell)_{\ell \in \mathbb Z}$ such that the cycle corresponding to each $e_\ell$ has non-trivial voltage.
\end{lem}
%
\begin{proof}Let $\gamma$ be the two-way hamiltonian path sketched in \Cref{aligned+volt} (b). Like in the proof of \Cref{align-fullgrid}, our goal is to recursively modify this hamiltonian path so that the voltage of the cycle corresponding to each jumping edge is non-trivial, and like in this proof (by symmetry) it is enough to consider jumping edges $e_i$ for $i > 0$. We use the same notation as in the proof of \Cref{align-fullgrid}: let $e_i$ be the $i$-th jumping edge, let $P_i$ be the subpath of $\gamma$ connecting the endpoints of $e_i$, let $C_i$ be the cycle consisting of $P_i$ and $e_i$, and let $P_{i,j}$ and $C_{i,j}$ be the corresponding path and cycle after $j$ modification steps.

Let $f_0, f_1, f_2$ be the set of edges of a strongly useful $\y$-triple in the finite factor of the grid. The modification step is sketched in \Cref{finitewidth}. 

\begin{figure}
    \centering
    \subfloat[\centering ]{{\includegraphics[scale=0.7]{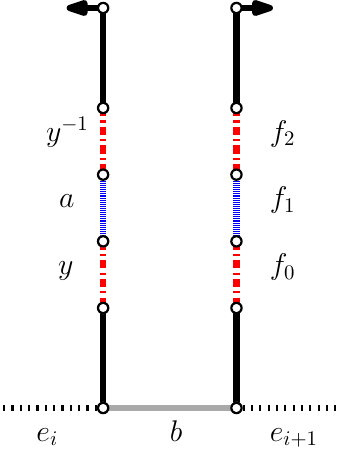} }}%
    \qquad
        \subfloat[\centering ]{{\includegraphics[scale=0.7]{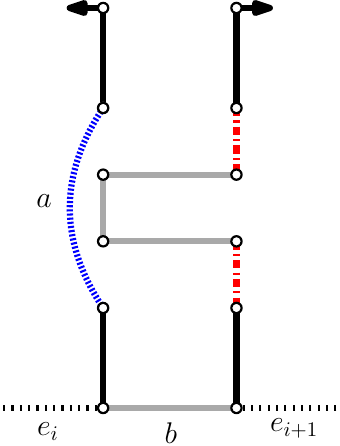} }}%
    \qquad
        \subfloat[\centering ]{{\includegraphics[scale=0.7]{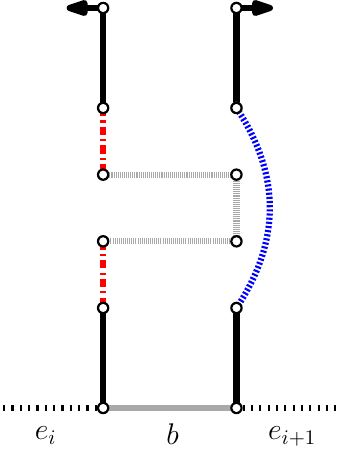} }}%

    \caption{Modifying $C_{j,j-1}$ by adding or removing a square. (instead of $e_i$ it should be $e_j$ to match the description in the text)}
    \label{finitewidth}
\end{figure}

Similarly to the proof of \Cref{align-fullgrid}, we ensure that $\vol(C_{j,j}) \neq 1$, and $C_{i,j} = C_{i,i}$ for every $i > j$. Moreover, we will ensure that all vertices of $C_i$ are contained in $C_{i',j}$ for $i' \in \{i-1,i,i+1\}$, and the same argument as in this proof then shows that replacing every $P_i$ by $P_{i,i}$ yields the desired hamiltonian path.

Clearly the properties hold before the first iteration, as there is no $C_{i,i}$ whose voltage needs to be non-trivial. Inductively assume that they hold after $j-1$ iterations for some $j > 0$.

If $\vol(C_{j,j-1}) \neq 1$, we can set $C_{i,j} = C_{i,j-1}$, and all properties are satisfied after the $j$-th iteration. So assume that $C_{j,j-1}$ has trivial voltage.

For $i \notin \{j,j+1\}$, we set $P_{i,j} = P_{i,j-1}$. If $\Gamma$ has a finite $\x$-$\y$-component, then obtain $P_{j,j}$ and $P_{j+1,j}$ by replace the parts of $P_{j,j-1}$ and $P_{j+1,j-1}$ shown in \Cref{finitewidth} (a) by the ones shown in \Cref{finitewidth} (b). If $\Gamma$ has an infinite $\x$-component, then we replace the parts of $P_{j,j-1}$ and $P_{j+1,j-1}$ shown in \Cref{finitewidth} (a) by the ones shown in \Cref{finitewidth} (c). In both cases it can be shown (analogously to the last part of the proof of \Cref{align-fullgrid}) that $\vol(C_{j,j})$ is conjugate to $[x,y]$, and thus non-trivial.
\end{proof}
\begin{nota}[cf.~{\cite[\S2.1]{WitteGallian-survey}}]
For $v \in G$ and an infinite sequence $(s)=(\ldots,s_{-2},s_{-1},s_1, s_2, \ldots )$ of elements of $S^{\pm 1}$we define a walk $v(s)$ in $\cay(G;S)$ that visits the vertices:
\[\dots, vs_{-1}^{-1} s_{-2}^{-1}, vs_{-1}^{-1} , v, vs_1, vs_1s_2,\dots \]
Moreover, we use $(\ldots a,s^k,b,\ldots)$ as a shorthand for $(\ldots,a,\underbrace{s,\ldots,s}_{k},b,\ldots)$.
\end{nota}

\begin{lem}
\label{align-only_xy}
If $\g=\langle \x,\y\rangle$, then $\cay (\g;\s)$ contains a two-way hamiltonian path with an aligned sequence $(e_\ell)_{\ell \in \mathbb Z}$ such that the cycle corresponding to each $e_\ell$ has non-trivial voltage.
\end{lem}

\begin{proof}
Since $\g$ is infinite, we know that at least one of $\x$ and $\y$ has infinite order. Assume without loss of generality that the order of $\x$ is infinite. If there is no $n$ such that $\y^n \in \langle \x \rangle$, then $\cay (\g;\s)$ has a consistently labelled spanning grid of with an infinite $\x$-component and an infinite $\y$-component, then \cref{align-fullgrid} finishes the proof.

So assume that there is such an $n$, and let $n\in \mathbb N$ be minimal such that $\y^n \in \langle \x \rangle$. If $n=1$, then $\g = \langle \x \rangle$, so $\g$ is infinite cyclic, and $\y$ generates a subgroup of finite index $>1$. By swapping the roles of $\x$ and $\y$ we may thus assume that $n > 1$. Moreover, we may assume that $\y^n = \x^{-m}$ for some $m \geq 0$, otherwise replace $\x$ by $\x^{-1}$.

We note that $\cay(\g,\{\x,\y\})$ contains two hamiltonian paths whose edge labels are periodic with period $(x,y^{n-1},x,{y^{-n+1}})$ and $(y^{n-1},x^{m+1})$, respectively. Both of these have aligned sequences of jumping edges, see \cref{aligned+volt}(b) and \Cref{hamilton+jumping}.

\begin{figure}
    \centering
    \includegraphics[scale=0.5]{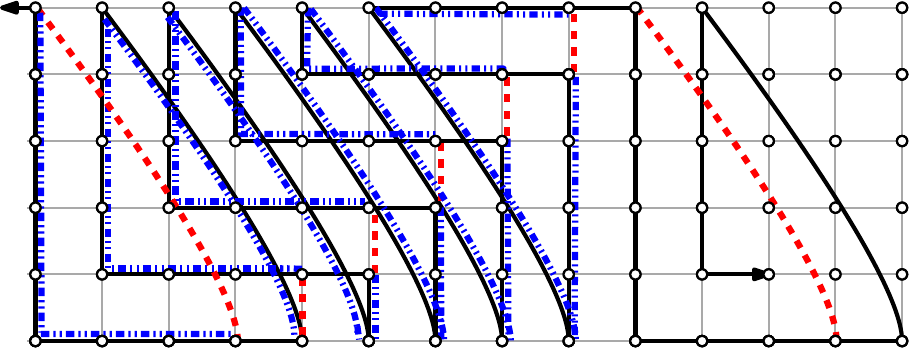}
    \caption{The illustration of a two-way hamiltonian path in $\cay(\g,\{\x,\y\})$ with jumping edges for the case $m=3$ and $n=6$.}
    \label{hamilton+jumping}
\end{figure}

The voltage of the cycles corresponding to jumping edges in \Cref{aligned+volt} (a) is $[y^{n-1},x]$, the voltage of the cycles in \Cref{hamilton+jumping} is $x^m y^n$. If one of these voltages is non-trivial, then we are done.

So assume that both voltages are trivial. From $x^m y^n = 1$ we obtain (by conjugation) $x^{m-1}y^nx = 1$ and $y^{n-1}x^my = 1$. Since $y^{n-1}$ and $x$ commute, we have that $y^{n-1}x^my = x^{m-1}y^{n-1}xy$, and thus
\[
    x^{m-1}y^nx = x^{m-1}y^{n-1}xy.
\]
Cancelling $x^{m-1}y^{n-1}$ on both sides of this equation we get $xy = yx$, contradicting the assumption that $[x,y] \neq 1$.
\end{proof}


\begin{thm}\label{always_align}
Let $G = \langle S \rangle$ be a group, let $\g = G/G'$ and let $\s$ be the generating set of $\g$ corresponding to $S$.
Then $\cay(\g;\s)$ has a two-way hamiltonian path $\gamma$ aligned with sequence $(e_{\ell})_{\ell\in \Z}$ such that $\vol (C_{\ell})\neq 1$  for every $\ell\in\Z$.
\end{thm}
\begin{proof}
    If $\x = \y$, then \Cref{align-coincide} applies. Otherwise, by \Cref{spanning-double} above, one of \Cref{align-fullgrid}, \Cref{align-finitegrid}, and \Cref{align-only_xy} applies.
\end{proof}


\section{Two-way hamiltonian paths in transitive \texorpdfstring{$G$}{G}-graphs with \texorpdfstring{$G' \cong \Z_p$}{}}

In this section, we first prove $\cay(G;S)$ has a two-way hamiltonian path.
Throughout the remainder of the paper, $p$ always denotes a prime number.


\begin{lem}\label{S_not_intersect_G'}
Let $G=\langle S\rangle$ be a  group such that  $G'\cap S=\emptyset$ and $G'\cong \Z_{p}$.
Then $\cay(G;S)$ has a two-way hamiltonian path.
\end{lem}
\begin{proof}
It follows from \cref{always_align} that $\cay(\g;\s)$ has a two-way hamiltonian path $\gamma$ aligned with a sequence $[e_{\ell}\mid \ell\in \Z]$ such that $\vol (C_{\ell})\neq 1 $  in $ G $.
Let $\gamma\colon \ldots,v_1,v_0,v_1,\ldots$ be a two-way hamiltonian path in $\cay(\g;\s)$ aligned with a sequence $[e_{\ell}\mid {\ell}\in \Z]$ such that $\vol (C_{\ell})\neq 1$ for each $\ell\in \Z$.
By \Cref{blocks_have_HC}, for every vertex at the entrance of a block $B$, there exists a hamiltonian path that starts at that vertex and ends at the exit of $B$. Additionally, there is a perfect matching between the entrance and the exit vertices of two consecutive blocks, allowing us to connect two hamiltonian paths in consecutive blocks, see \Cref{enter-label}.
\end{proof}

\begin{lem}\label{S_intersect_G'}
Let $G=\langle S\rangle$ be a group such that  $G'\cap S\neq\emptyset$ and $G'\cong \Z_p$. 
Then $\cay(G;S)$ has a two-way hamiltonian path.
\end{lem}
\begin{proof}
First, let's define $K\coloneqq \langle G'\cap S\rangle $. 
Since $ G' $ is cyclic, $ K $ becomes a characteristic subgroup of $ G' $, thus $ K \unlhd G $. As $ G/K \cap SK = \emptyset $, according to \Cref{S_not_intersect_G'}, there exists a two-way hamiltonian path $ \overline{\gamma} \colon {\ldots,u_{-1},u_{0},u_{1},\ldots} $ in $ \cay(G/K;SK) $. 
We note that $\cay(K;G'\cap S)$ has a hamiltonian cycle, as shown in \Cref{abelian_ham} (I).
Consequently, there exists a hamiltonian cycle $C_i\colon v_1^i,\ldots,v_m^i$ in each $\pi^{-1}(u_i)$. 
A clockwise orientation is fixed in each $C_i$. Let $p_i$ represent the hamiltonian path starting with $v^i_1$ and moving along $C_i$ clockwise. 
There is a  perfect matching between $\pi^{-1}(u_i)$ and $\pi^{-1}(u_{i+1})$. Subsequently, an edge $e_i$ connecting $v^i_n$ to a vertex $v^i_{\ell}$ in $\pi^{-1}(u_{i+1})$ is used. Next, we find the hamiltonian path $p_{i+1}$ along the cycle $C_{i+1}$ clockwise. This procedure is repeated for each $i' \in \mathbb{Z}$, as depicted in \Cref{spiral}. It can be seen that $\{ p_i\mid i\in\mathbb{Z}\}\bigcup \{e_i\mid i\in \mathbb{Z}\}$ constitutes a two-way hamiltonian path.
\end{proof}
\begin{figure}[H]
    \centering
    \includegraphics[scale=0.6]{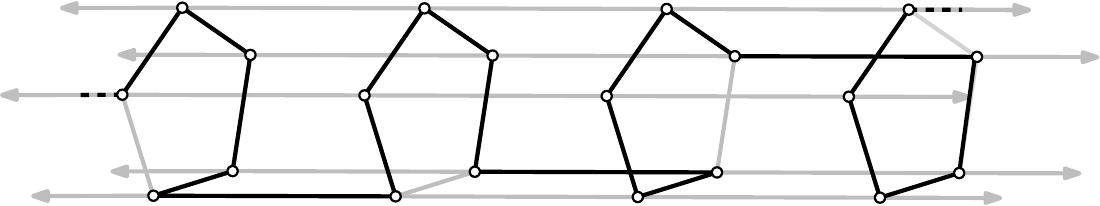}
    \caption{An illustration of a two-way hamiltonian path in \cref{S_intersect_G'} for the case  $p = 5$.}
    \label{spiral}
\end{figure}



\maintwo*

\begin{proof}
The proof follows from \Cref{S_not_intersect_G'} and \Cref{S_intersect_G'}.
\end{proof}

Next, we extend \cref{main_2} to transitive graphs whose automorphism group has a finite cyclic group of order $p$ as its commutator subgroup. However, we first need to review a few preliminary results.

\begin{lem}[Sabidussi's Theorem {\cite[Theorem 1.2.20]{dobson2022symmetry}}]\label{sabidussi}
Let $X$ be a transitive $G$-graph.
There is a subset~$S$ of~$G$, such that $X$ is isomorphic to the Cayley graph\/ $\cay(G; S)$ if and only if\/ $\mathsf{Aut}(X)$ contains a subgroup that is isomorphic to~$G$ and acts sharply transitively on $V(X)$.
\end{lem}
In particular, if $G$ acts on a graph $X$ transitively such that $\st_G(x)$ is trivial for some $x\in V(X)$, then $X$ is (isomorphic to) a Cayley graph on $G$.
We note that $G/G'$ is abelian which implies that $G'\st(v)/G'$ is a normal subgroup of $G/G'$ for every $v\in V(X)$.
So $G'\st(v)$ is a normal subgroup of $G$ for every $v\in V(X)$ and we have the following corollary:
\begin{cor}\label{quotient}
Let $X$ be a transitive $G$-graph. 
Then $X/G'$ is a Cayley graph of $G/(G'\st_G(v))$.
\end{cor}


\begin{lem}\label{cor}
Let $X$ be a $G$-graph and $v\in V(X)$.
If $G$ acts on $X$ transitively but not sharply. 
There exists a neighbor $u$ of $v$ such that $\st_G(v)\neq \st_G(u)$.\label{adjacent} 
\end{lem}

\begin{proof}
If the stabilizer of each neighbor $ y $ of a vertex $ x $ is equal to the stabilizer of $ x $, then there exists an automorphism that fixes $ X $, implying that the action is not faithful. Therefore, there must be a vertex $ x $ with a neighbor $ y $ such that $ \st_G(x) \neq \st_G(y) $.
Since $ G $ is transitive, we conclude that there exists a neighbor $ u $ of $ v $ such that $ \st_G(v) \neq \st_G(u) $.
\end{proof}

\mainone*

\begin{proof}
If the stabilizer of a vertex of $X$ is trivial, then $X$ is a Cayley graph and by \cref{main_2}, we are done.
So we can assume that there is a vertex with non-trivial stablizer.
It follows from \cref{quotient} that $X/G'$ is a Cayley graph of $G/(G'\st_G(v))$.
We notice that $G/(G'\st_G(v))$ is an abelian group.
There is  a two-way hamiltonian  path $\overline{\gamma}$ aligned with the sequence  $[e_{\ell}\mid \ell\in \Z]$.
Let $G'=\{g_1,\ldots,g_n\}$ and $g_1=1$.
We set  two-way paths $L(\overline{\gamma})\coloneqq \{\gamma_1,\ldots,\gamma_{n}\}$ of liftings of $\overline{\gamma}$,
where $\gamma_i$ is 
\[ \ldots, g_iv_{-2}, g_iv_{-1}, g_iv_0, g_iv_{1}, g_iv_{2}, \ldots\]
We note that by \Cref{|N|-liftings} we have $V(X)=\sqcup_{i=1}^p V(R_i)$.
We claim the following
\begin{clm}
If $\pi((u,v))=e_{\ell}$, then there is a bouncing edge in $\pi^{-1}(e_{\ell})$.
\end{clm}
\begin{proof}
Let $(G'[u], G'[v])=e_{\ell}$ be an edge such that the lifting edge $(u,v)$ is not bouncing.
By \cref{cor}, there is a neighbor $v$ of $u$ such that $\st_G(u)\neq\st_G(v)$ and so there is $\zeta\in G$ such that $\zeta(u)=u$, but $\zeta(v)\neq v$.
There exists $\alpha\in G$ such that $\alpha u=v$ and so $G'\st_G(v)=(G'\st_G(u))^{\alpha}$.
As we mentioned before  $G'\st(u)$ is a normal subgroup of $G$.
Thus $G'\st_G(v)=G'\st_G(u)$ which implies that $\zeta(v)\in G'[v]$.
We note that $u$ and $v$ are adjacent and so  $\zeta(u)=u$ and $\zeta(v)$ are adjacent.
We are assuming that $(u,v)$ is not bouncing. Since $\zeta(v)\neq v$ and $\zeta(v)\in G'[v]$, we conclude that $(u,\zeta(v))$ is a bouncing edge and $\pi(u,\zeta(v))=e_{\ell}$.
\end{proof}

\noindent Since for each $e_{\ell}$ for $\ell \in \Z$, there is a bouncing edge in $\pi^{-1}(e_{\ell})$, we are able to setup our blocks in $X$, as described in \Cref{blocks}.
By \Cref{blocks_have_HC}, for every vertex at the entrance of a block $B$, there exists a hamiltonian path that starts at that vertex and ends at the exit of $B$. Additionally, there is a perfect matching between the entrance and the exit vertices of two consecutive blocks, allowing us to connect two hamiltonian paths in consecutive blocks.
\end{proof}


\section{Hamiltonian circles}\label{circles}

Diestel and Kühn (cf.~\cite{Diestel-CycleSpace}) have introduced a topological point of view that suggests a different analogue of a hamiltonian cycle for infinite graphs that are locally finite. The inspiration behind this approach is as follows:

\noindent In a finite graph, a hamiltonian cycle corresponds to a copy of the circle $S^1$ in $X$, which includes all the vertices of $X$.
In the case of an infinite graph $X$, adding the point at infinity to a two-way-infinite hamiltonian path results in a copy of $S^1$ in the one-point compactification of $X$. This compactification contains all the vertices of $X$.
The new approach acknowledges that circles in a different compactification, specifically the Freudenthal compactification, can be more interesting in the context of infinite graphs.

\noindent If an infinite graph has only one end, then its Freudenthal compactification is the same as its one-point compactification. In such cases, it is easy to observe that every hamiltonian circle is a two-way-infinite hamiltonian path.

\begin{cor}
Let $\cay(G;S)$ be a one-ended graph, where $G'\cong\Z_p$.
Then $\cay(G;S)$ has a hamiltonian circle.
\end{cor}

\noindent We close the paper with the following conjecture.
\begin{conj}
Let $G=\langle S\rangle$ be a two-ended group, where $G'\cong \Z_{p^n}$.
Then $\cay(G;S)$ has a hamiltonian circle.
\end{conj}

\noindent Also one can ask if $G'\cong \Z_{p^n}$, does $G=\langle S\rangle$ have a two-way hamiltonian path?

\section*{Acknowledgement}
We thank Dave Witte Morris for a discussion during the preparation of this paper.

\bibliographystyle{plainurlnat}
\bibliography{hc.bib}
\nocite{*}

\end{document}